\documentclass[twoside,11pt,reqno]{amsart}
\usepackage{amsmath, amsthm,amssymb,amstext,amsfonts,amscd}
\usepackage{graphicx}
\usepackage{multirow}
\usepackage{lmodern}
\usepackage{latexsym}
\usepackage{wasysym} %\usepackage{fontspec}
\usepackage{url}
\DeclareMathOperator{\erf}{erf}
\newtheorem{theorem}{Theorem}

\numberwithin{equation}{section}

\begin{document}

\title[{General series identities, some additive theorems on hypergeometric functions}]
{General series identities, some additive theorems on hypergeometric functions and their applications}

\author[M.I. Qureshi, S. Jabee and M. Shadab]{Mohammad Idris Qureshi, Saima Jabee and Mohd Shadab$^{*}$}

\address{Mohammad Idris Qureshi: Department of Applied Sciences and Humanities,
Faculty of Engineering and Technology,
Jamia Millia Islamia (A Central University),
New Delhi 110025, India.}
\email{miqureshi\_delhi@yahoo.co.in}

\address{Saima Jabee: Department of Applied Sciences and Humanities,
Faculty of Engineering and Technology,
Jamia Millia Islamia (A Central University),
New Delhi 110025, India.}
\email{saimajabee007@gmail.com}

\address{Mohammad Shadab: Department of Applied Sciences and Humanities,
Faculty of Engineering and Technology,
Jamia Millia Islamia (A Central University),
New Delhi 110025, India.}
\email{shadabmohd786@gmail.com}



\thanks{*Corresponding author}

\begin{abstract}
Motivated by the substantial development of the special functions,
we contribute to establish some rigorous results on the general series identities with bounded sequences and hypergeometric functions with different arguments, which are generally applicable in nature.
For the application purpose, we apply our results to some functions
e.g. Trigonometric functions, Elliptic integrals, Dilogarithmic function, Error function, Incomplete gamma function,
 and many other special functions.
\end{abstract}

\maketitle

\section{Introduction, Preliminaries and Notations}
In present paper, we shall use the following standard notations:\\
$\mathbb{N} :=\{1,2,3,\ldots \}$,\\
$\mathbb {N}_0:=\{0,1,2,3,\ldots \}=\mathbb{N}\cup \{0\}$,\\
$\mathbb {Z}_0^- :=\{0,-1,-2,-3,\ldots \}$,\\
$\mathbb {Z}^- :=\{-1,-2,-3,\ldots \}=\mathbb {Z}_0^- \backslash \{0\},$ and\\
$\mathbb{Z}=\mathbb {Z}_0^-\cup \mathbb{N}$.\\
Here, as usual, $\mathbb {Z}$ denotes the set of integers, $\mathbb{R}$ denotes the set of real numbers,  $\mathbb{R^+}$ denotes the set of positive real numbers and $\mathbb{C}$ denotes the set of complex numbers.\\
The Pochhammer symbol (or the shifted factorial) $(\lambda)_{\nu}$ $(\lambda,\nu \in \mathbb{C})$ is defined in terms of the familiar Gamma function, by
\begin{equation}\label{h.eq(1.1)}
(\lambda)_{\nu}:=\frac{\Gamma(\lambda+\nu)}{\Gamma(\lambda)}=
\begin{cases}
$1$ & ;(\nu=0;~ \lambda \in \mathbb {C}\backslash \{0\})\\
\lambda (\lambda+1)\ldots (\lambda+n-1) & ;~(\nu=n \in \mathbb {N}; ~\lambda \in \mathbb {C}\backslash {\mathbb{Z}_0^-})\\
\frac{(-1)^{n}k!}{(k-n)!} & ;~(\lambda=-k; ~\nu=n; ~n,k \in \mathbb{N}_0;~ {0}\leq{n}\leq{k})\\
$0$ & ;~(\lambda=-k; ~\nu=n; ~n,k \in \mathbb{N}_0;~{n}>{k})\\
\frac{(-1)^{n}}{(1-\lambda)_n} & ;~(\nu=-n; n\in \mathbb{N}; ~\lambda \neq 0,\pm{1},\pm{2},\ldots),
\end{cases}
\end{equation}
it being understood {\it conventionally} that $(0)_0=1$, and assumed tacitly that the Gamma quotient exists.\\
In the Gaussian hypergeometric series $_2F_1(a,b;c;z)$, there are two numerator parameters $a$, $b$ and one denominator parameter $c$. A natural generalization of this series is accomplished by introducing any arbitrary number of numerator and denominator parameters. The non-terminating hypergeometric series \cite[p.42-43]{Srivastava3}
\begin{equation}\label{h.eq(1.2)}
{_pF_q}\left[\begin{array}{c}\alpha_1,\ldots,\alpha_p;\\\beta_1,\ldots,\beta_q;\end{array}z\right]=\sum_{n=0}^{\infty}\frac{(\alpha_1)_n\ldots(\alpha_p)_n}{(\beta_1)_n\ldots(\beta_q)_n}\frac{z^n}{n!},
\end{equation}
is known as the {\it generalized Gauss} and {\it Kummer series}, or simply, the {\it generalized hypergeometric series}. Here $p$ and $q$ are positive integers or zero (interpreting an empty product as unity), and we assume that the variable $z$, the numerator parameters ${\alpha_1,\ldots,\alpha_p}$ and the denominator parameters ${\beta_1,\ldots,\beta_q}$ take on complex values, provided that
\begin{equation}\label{h.eq(1.3)}
\beta_j\neq0,-1,-2,\ldots; \quad{j=1,\ldots,q}.
\end{equation}
Convergence conditions \cite[p.43]{Srivastava3} for generalized hypergeometric function are as follows:\\
Suppose that none of the numerator parameters is zero or a negative integer (otherwise the question of convergence will not arise), and with the usual restriction \eqref{h.eq(1.3)}, the ${_pF_q}$ series in the definition \eqref{h.eq(1.2)}\\
(i) converges for ${|z|<\infty}$, if $p~{\leqq}~q$,\\
(ii) converges for ${|z|<1}$, if $p~=~q+1$.\\
Furthermore, if we denote
\begin{equation}
\omega=\sum_{j=1}^{q}{\beta_j}-\sum_{j=1}^{p}{\alpha_j},\nonumber
\end{equation}
it is known that the ${_pF_q}$ series, with $p=q+1$, is\\
(a) absolutely convergent for $|z|=1$, if $\Re(\omega)>0$,\\
(b) conditionally convergent for $|z|=1$, $z\neq1$, if $-1<\Re(\omega) \leqq0$.\\
The Fox-Wright psi function of one variable (\cite[p.389]{Fox1}; see also \cite{Fox2,Wright1,Wright2}) is given by
\begin{eqnarray}\label{h.eq(1.4)}
{_p\Psi_q}\left[\begin{array}{c}(\alpha_1,A_1),\ldots,(\alpha_p,A_p);\\(\beta_1,B_1),\ldots,(\beta_q,B_q);\end{array}z\right]
&=&\sum_{n=0}^{\infty}\frac{\Gamma{(\alpha_1+A_1n)}\ldots\Gamma{(\alpha_p+A_pn)}}{\Gamma{(\beta_1+B_1n)}\ldots\Gamma{(\beta_q+B_q n)}}\frac{z^n}{n!}\nonumber\\
&=&\frac{\Gamma{(\alpha_1)}\ldots\Gamma{(\alpha_p)}}{\Gamma{(\beta_1)}\ldots\Gamma{(\beta_q)}}\sum_{n=0}^{\infty}\frac{{(\alpha_1)}_{nA_1}\ldots{(\alpha_p)}_{nA_p}}{{(\beta_1)}_{nB_1}\ldots{(\beta_q)}_{nB_q}}\frac{z^n}{n!},\nonumber\\
\end{eqnarray}
\begin{eqnarray}\label{h.eq(1.5)}
{_p\Psi_q}\left[\begin{array}{c}(\alpha_1,A_1),\ldots,(\alpha_p,A_p);\\(\beta_1,B_1),\ldots,(\beta_q,B_q);\end{array}z\right]&=&\frac{\Gamma{(\alpha_1)}\ldots\Gamma{(\alpha_p)}}{\Gamma{(\beta_1)}\ldots\Gamma{(\beta_q)}}{_p\Psi_q^*}\left[\begin{array}{c}(\alpha_1,A_1),\ldots,(\alpha_p,A_p);\\(\beta_1,B_1),\ldots,(\beta_q,B_q);\end{array}z\right],\nonumber\\
\end{eqnarray}
\begin{eqnarray}\label{h.eq(1.7)}
{_p\Psi_q}\left[\begin{array}{c}(\alpha_1,A_1),\ldots,(\alpha_p,A_p);\\(\beta_1,B_1),\ldots,(\beta_q,B_q);\end{array}z\right]&=&\frac{1}{2\pi\rho}\int_{L}\frac{{\Gamma(\zeta)}\displaystyle\prod_{i=1}^{p}\Gamma{(\alpha_{i}-A_{i}\zeta)}}{\displaystyle\prod_{j=1}^{q}\Gamma{(\beta_{j}-B_{j}\zeta)}}(-z)^{-\zeta}d\zeta,
\end{eqnarray}
where $\rho^2=-1$, $z\in\mathbb{C}$; parameters $\alpha_{i},\beta_{j}\in\mathbb{C}$; coefficients $A_{i},B_{j}\in\mathbb{R}=(-\infty,+\infty)$ in case of series \eqref{h.eq(1.4)} (or $A_{i},B_{j}\in\mathbb{R}_{+}=(0,+\infty)$ in case of contour integral \eqref{h.eq(1.7)}), $A_{i}\neq0~(i=1,2,...,p), B_{j}\neq0~(j=1,2,...,q)$. In equation \eqref{h.eq(1.4)}, the parameters $\alpha_{i},\beta_{j}$ and coefficients $A_{i},B_{j}$ are adjusted in such a way that the product of Gamma functions in numerator and denominator should be well defined \cite{Boersma,Braaksma}.
\begin{eqnarray}
\Delta^{*}=\left( \sum_{j=1}^{q}B_{j}-\sum_{i=1}^{p}A_{i}\right),\\
\delta^{*}=\left(\displaystyle\prod_{i=1}^{p}|A_{i}|^{-A_{i}}\right)\left(\displaystyle\prod_{j=1}^{q}|B_{j}|^{B_{j}}\right),\\
\mu^{*}= \sum_{j=1}^{q}\beta_{j}-\sum_{i=1}^{p}\alpha_{i}+\left(\frac{p-q}{2}
\right),
\end{eqnarray}
and
\begin{eqnarray}
\sigma^{*}=(1+A_{1}+...+A_{p})-(B_{1}+...+B_{q})=1-\Delta^{*}.
\end{eqnarray}
Case(I): When contour (L) is a left loop beginning and ending at $-\infty$, then ${_p}\Psi_{q}$ given by \eqref{h.eq(1.4)} or \eqref{h.eq(1.7)} holds the following convergence conditions\\
i) When $\Delta^{*}>-1, 0<|z|<\infty, z\neq0,$\\
ii) When $\Delta^{*}=-1, 0<|z|<\delta^{*},$\\
iii)When $\Delta^{*}=-1, |z|=\delta^{*}, \text{ and } \Re(\mu^{*})>\frac{1}{2}.$\\
Case(II): When contour (L) is a right loop beginning and ending at $+\infty$, then ${_p}\Psi_{q}$ given by \eqref{h.eq(1.4)} or \eqref{h.eq(1.7)} holds the following convergence conditions\\
i) When $\Delta^{*}<-1, 0<|z|<\infty, z\neq0,$\\
ii) When $\Delta^{*}=-1, |z|>\delta^{*},$\\
iii)When $\Delta^{*}=-1, |z|=\delta^{*}, \text{ and } \Re(\mu^{*})>\frac{1}{2}.$\\
Case(III): When contour (L) is starting from $\gamma-i\infty$ and ending at $\gamma+i\infty$, where $\gamma\in\mathbb{R}=(-\infty,+\infty)$, then ${_p}\Psi_{q}$ is also convergent under the following conditions\\
i) When $\sigma^{*}>0, |\arg(-z)|<\frac{\pi}{2}\sigma^{*}, 0<|z|<\infty, z\neq0,$\\
ii) When $\sigma^{*}=0, \arg(-z)=0, 0<|z|<\infty, z\neq0 \text{ such that } -\gamma\Delta^{*}+\Re(\mu^{*})>\frac{1}{2}+\gamma,$\\
iii)When $\gamma=0,\sigma^{*}=0,\arg(-z)=0, 0<|z|<\infty, z\neq0 \text{ such that }, \Re(\mu^{*})>\frac{1}{2}.$\\
Next we collect some results that we will need in the sequel.\\

\vskip.1cm
{\bf Identity 1.} Let $$ \alpha=\exp{\left( \frac{2 \pi i}{5} \right)}, i=\sqrt{(-1)}$$ and $r$ being non-negative integer, then
\begin{eqnarray}
1+ \alpha^{2r} + \alpha^{4r} + \alpha^{6r} + \alpha^{8r} = \begin{cases}
\begin{array}{c}
5;\\
~\\
0;
\end{array} & \begin{array}{c}
r \in \{0,5,10,15,20,\dots\}\\
~\\
r \in \{1,2,3,4,6,7,8,9,11,12,13,14,\dots\}.
\end{array}\end{cases}
\end{eqnarray}

\vskip.1cm

{\bf Identity 2.} Let $$ \alpha=\exp{\left( \frac{2 \pi i}{5} \right)}, i=\sqrt{(-1)}$$ and $r$ being non-negative integer, then
\begin{eqnarray}
&&1+ \alpha^{2r+1} + \alpha^{4r+2} + \alpha^{6r+3} + \alpha^{8r+4}\nonumber\\
&&\hskip15mm=\begin{cases}
\begin{array}{c}
5;\\
~\\
0;
\end{array} & \begin{array}{c}
r \in \{2,7,12,17, \dots \}\\
~\\
r \in \{ 0,1,3,4,5,6,8,9,10,11,13,14,15,16,\dots \}.
\end{array}\end{cases}
\end{eqnarray}
Above identities can be verified with the help of De Moivre's theorem and some trigonometrical identities.\\

{\bf Gauss multiplication formula.} Let m being positive integer and n being non-negative integer, then
\begin{eqnarray}\label{h.inteq(1.7)}
(b)_{mn}= m^{mn}\displaystyle\prod_{j=1}^{m} \Big(\frac{b+j-1}{m}\Big)_{n}.
\end{eqnarray}

Now, we are recalling some functions in the hypergeometric notations \cite[pp. 71, 115]{Rainville} (see also, \cite{Gradshteyn}), which we will use in the applications.

\begin{table}[h!]
\centering

\caption{Some elementary functions and its hypergeometric representations}
\vskip.1cm
\label{table-1.2}
\begin{tabular}{|c  |c  |c  |c|c|c|c|c|c|}
\hline
Ser. No. & Notation  & Hypergeometric Representation \\ [1pt]\hline

1 & $\arcsin(x)$ & $\arcsin(x)= x\,{_{2}F_{1}}\left[\begin{array}{r} \frac{1}{2},\frac{1}{2};\\~\\ \frac{3}{2};\end{array} x^2  \right]$  \\ [15pt]\hline

2 & $\arctan(x)$ &  $\arctan(x) = x\,{_{2}F_{1}}\left[\begin{array}{r} 1, \frac{1}{2};\\~\\ \frac{3}{2};\end{array} -x^{2} \right]$\\ [15pt]\hline

3 & $\sin(x)$ &  $\sin(x) =  x\,{_{0}F_{1}}\left[\begin{array}{r} -;\\~\\ \frac{3}{2};\end{array} -\frac{x^{2}}{4} \right]$\\ [15pt]\hline

4 & $(\arcsin(x) )^2$ &  $(\arcsin(x) )^2 = x^2 \,{_{3}F_{2}}\left[\begin{array}{r} 1,1,1;\\~\\ 2, \frac{3}{2};\end{array} x^{2} \right]$\\ [15pt]\hline

5 & $\cos(x)$ &  $\cos(x) = {_{0}F_{1}}\left[\begin{array}{r} -;\\~\\ \frac{1}{2};\end{array} -\frac{x^{2}}{4} \right]$\\ [15pt]\hline

5 & $\left( \frac{2}{1+\sqrt{(1-x^2)}} \right)^{2\gamma-1}$ & $\left( \frac{2}{1+\sqrt{(1-x^2)}} \right)^{2\gamma-1}= {_{2}F_{1}}\left[\begin{array}{r} \gamma,\gamma-\frac{1}{2};\\~\\ 2\gamma;\end{array} x^2  \right]$\\ [15pt]\hline

\end{tabular}
\end{table}

\begin{table}[h!]
\centering

\caption{Some special functions and its hypergeometric representations}
\vskip.1cm
\label{table-1.2}
\begin{tabular}{|c  |c  |c  |c|c|c|c|c|c|}
\hline
Ser. No. & Notation  & Hypergeometric Representation   \\ [1pt]\hline
1 & Complete elliptic integral of first kind: $\textbf{K}(x)$ & $\textbf{K}(x)=\frac{\pi}{2} \,{_{2}F_{1}}\left[\begin{array}{r} \frac{1}{2}, \frac{1}{2};\\~\\ 1;\end{array} x^{2} \right]$  \\ [5pt]\hline

2 & Complete elliptic integral of second kind: $\textbf{E}(x)$ & $\textbf{E}(x)=\frac{\pi}{2} \,{_{2}F_{1}}\left[\begin{array}{r} -\frac{1}{2}, \frac{1}{2};\\~\\ 1;\end{array} x^{2} \right]$  \\ [5pt]\hline

3 & Error function or Probability integral: $\erf(x)$ & $\erf(x)=\frac{2x}{\sqrt(\pi)} \,{_{1}F_{1}}\left[\begin{array}{r}  \frac{1}{2};\\~\\ \frac{3}{2};\end{array} -x^{2} \right]$  \\ [5pt]\hline

4 & Incomplete gamma function:  $\gamma(a,x)$ & $\gamma(a,x^2)=\frac{x^{2a}}{a} \,{_{1}F_{1}}\left[\begin{array}{r}  a;\\~\\ 1+a;\end{array} -x^2 \right]$  \\ [5pt]\hline

 5& Dilogarithm function:  ${Li}_2(x)$ & ${Li}_2(x^2) = x^2 \,{_{3}F_{2}}\left[\begin{array}{r} 1,1,1;\\~\\ 2, 2;\end{array} x^2 \right]$  \\ [5pt]\hline
\end{tabular}
\end{table}

\section{General Series Identities}
\begin{theorem}\label{h.thm2.1}
Suppose $\{\phi(r) \}_{r=0}^{\infty}$ is a bounded sequence of arbitrary real and complex numbers and  $$ \alpha=\exp{\left( \frac{2 \pi i}{5} \right)}, i=\sqrt{(-1)}$$ then
\begin{eqnarray}\label{h.eq(2.1)}
&&\sum_{r=0}^{\infty}\phi(r) \frac{c^r (x)^{2r}}{r!} + \sum_{r=0}^{\infty}\phi(r) \frac{c^r (x\alpha)^{2r}}{r!} + \sum_{r=0}^{\infty}\phi(r) \frac{c^r (x\alpha^2)^{2r}}{r!} + \sum_{r=0}^{\infty}\phi(r) \frac{c^r (x\alpha^3)^{2r}}{r!} \nonumber\\
&& \hskip20mm + \sum_{r=0}^{\infty}\phi(r) \frac{c^r (x\alpha^4)^{2r}}{r!}=5 \sum_{r=0}^{\infty}\phi(5r) \frac{c^{5r} x^{10r}}{ (5r)!},
\end{eqnarray}
provided that each of the series involved is absolutely convergent.
\end{theorem}

\begin{proof} Suppose LHS of equation \eqref{h.eq(2.1)} is denoted by $S$, then
\begin{eqnarray}\label{h.eq(2.2)}
&&S=\sum_{r=0}^{\infty}\phi(r) \frac{c^r x^{2r}}{r!} \{1+ \alpha^{2r} + \alpha^{4r} + \alpha^{6r} +  \alpha^{8r} \}.
\end{eqnarray}

Now, we apply Identity 1 in equation \eqref{h.eq(2.2)}, we get
\begin{eqnarray}\label{h.eq(2.3)}
&&S=5\, \phi(0) \frac{c^{0} x^{0}}{ (0)!} +  5\,\phi(5) \frac{c^{5} x^{10}}{ (5)!} + 5\,\phi(10) \frac{c^{10} x^{20}}{ (10)!}
+ 5\,\phi(15) \frac{c^{15} x^{30}}{ (15)!}\, + \dots\nonumber\\
&&\hskip20mm  =5 \sum_{r=0}^{\infty}\phi(5r) \frac{c^{5r} x^{10r}}{ (5r)!}.
\end{eqnarray}

\end{proof}

\begin{theorem}\label{h.thm2.2}
Suppose $\{\phi(r) \}_{r=0}^{\infty}$ is a bounded sequence of arbitrary real and complex numbers and  $$ \alpha=\exp{\left( \frac{2 \pi i}{5} \right)}, i=\sqrt{(-1)}$$ then
\begin{eqnarray}\label{h.eq(2.4)}
&&\sum_{r=0}^{\infty}\phi(r) \frac{c^r (x)^{2r}}{r!} + \alpha \sum_{r=0}^{\infty}\phi(r) \frac{c^r (x\alpha)^{2r}}{r!} + \alpha^2 \sum_{r=0}^{\infty}\phi(r) \frac{c^r (x\alpha^2)^{2r}}{r!} +  \alpha^3 \sum_{r=0}^{\infty}\phi(r) \frac{c^r (x\alpha^3)^{2r}}{r!} \nonumber\\
&&\hskip20mm + \alpha^4 \sum_{r=0}^{\infty}\phi(r) \frac{c^r (x\alpha^4)^{2r}}{r!} =5 \sum_{r=0}^{\infty}\phi(5r+2) \frac{c^{(5r+2)} x^{(10r+4)}}{ (5r+2)!},
\end{eqnarray}
provided that each of the series involved is absolutely convergent.
\end{theorem}

\begin{proof}
Suppose LHS of equation \eqref{h.eq(2.4)} is denoted by $T$, then
\begin{eqnarray}\label{h.eq(2.5)}
T=\sum_{r=0}^{\infty}\phi(r) \frac{c^r x^{2r}}{r!} \{1+ \alpha^{2r+1} + \alpha^{4r+2} + \alpha^{6r+3} +  \alpha^{8r+4} \}.
\end{eqnarray}

Now, we apply Identity 2 in equation \eqref{h.eq(2.5)}, we get
\begin{eqnarray}\label{h.eq(2.6)}
&&T=5\, \phi(2) \frac{c^{2} x^{4}}{ (2)!} +  5\,\phi(7) \frac{c^{7} x^{14}}{ (7)!} + 5\,\phi(12) \frac{c^{12} x^{24}}{ (12)!}
+ 5\,\phi(17) \frac{c^{17} x^{34}}{ (17)!}\, + \dots\nonumber\\
&&\hskip20mm
=5 \sum_{r=1}^{\infty}\phi(5r-3) \frac{c^{(5r-3)} x^{(10r-6)}}{ (5r-3)!}.
\end{eqnarray}

Replacing $r$ by $r+1$ in equation \eqref{h.eq(2.6)}, we get
\begin{eqnarray}\label{h.eq(2.7)}
T=5 \sum_{r=0}^{\infty}\phi(5r+2) \frac{c^{(5r+2)} x^{(10r+4)}}{ (5r+2)!}.
\end{eqnarray}
\end{proof}
\section{Hypergeometric Representations }
Any values of parameters and variables leading to the results, which do not make sense, are tacitly excluded.
\begin{theorem}\label{h.thm3.1}
Following sum of Fox-Wright hypergeometric functions with different arguments holds true:
\begin{eqnarray}
&&{_p\Psi_q}\left[\begin{array}{r} (a_{1},A_1),\,(a_{2},A_2),\,\dots\, ,(a_{p},A_p);\\~\\ (b_{1},B_1),\,(b_{2},B_2),\,\dots\, ,(b_{q},B_q);\end{array}  c(x)^2 \right]\nonumber\\
&&+  {_p\Psi_q}\left[\begin{array}{r} (a_{1},A_1),\,(a_{2},A_2),\,\dots\, ,(a_{p},A_p);\\~\\ (b_{1},B_1),\,(b_{2},B_2),\,\dots\, ,(b_{q},B_q);\end{array}  c(x \alpha)^2 \right]\nonumber\\
&&+  {_p\Psi_q}\left[\begin{array}{r} (a_{1},A_1),\,(a_{2},A_2),\,\dots\, ,(a_{p},A_p);\\~\\ (b_{1},B_1),\,(b_{2},B_2),\,\dots\, ,(b_{q},B_q);\end{array}  c(x \alpha^2)^2 \right]\nonumber
\end{eqnarray}
\begin{eqnarray}\label{h.eq(3.1)}
&&+  {_p\Psi_q}\left[\begin{array}{r} (a_{1},A_1),\,(a_{2},A_2),\,\dots\, ,(a_{p},A_p);\\~\\ (b_{1},B_1),\,(b_{2},B_2),\,\dots\, ,(b_{q},B_q);\end{array}  c(x \alpha^3)^2 \right]\nonumber\\
&&+   {_p\Psi_q}\left[\begin{array}{r} (a_{1},A_1),\,(a_{2},A_2),\,\dots\, ,(a_{p},A_p);\\~\\ (b_{1},B_1),\,(b_{2},B_2),\,\dots\, ,(b_{q},B_q);\end{array}  c(x \alpha^4)^2 \right]\nonumber\\
&&= 5\, \displaystyle\prod_{i=1}^{4} \Gamma {\left(\frac{i}{5}\right)} {_p\Psi_{q+4}}\left[\begin{array}{r} (a_{1},5A_1),\,(a_{2},5A_2),\,\dots\, ,(a_{p},5A_p);\\~\\ \left(\frac{1}{5}, 1\right), \left(\frac{2}{5}, 1\right), \left(\frac{3}{5}, 1\right), \left(\frac{4}{5}, 1\right), (b_{1},5B_1),\,(b_{2},5B_2),\,\dots\, ,(b_{q},5B_q);\end{array} \left( \frac{cx^2}{5}\right)^5 \right],\nonumber\\
\end{eqnarray}
where $\alpha=\exp{\left( \frac{2 \pi i}{5} \right)}.$
\end{theorem}

\begin{theorem}\label{h.thm3.2}
Following sum of Fox-Wright hypergeometric functions with different arguments holds true:
\begin{eqnarray}\label{h.eq(3.4)}
&&{_p\Psi_q}\left[\begin{array}{r} (a_{1},A_1),\,(a_{2},A_2),\,\dots\, ,(a_{p},A_p);\\~\\ (b_{1},B_1),\,(b_{2},B_2),\,\dots\, ,(b_{q},B_q);\end{array}  cx^2 \right]\nonumber\\
&&+ (\alpha) {_p\Psi_q}\left[\begin{array}{r} (a_{1},A_1),\,(a_{2},A_2),\,\dots\, ,(a_{p},A_p);\\~\\ (b_{1},B_1),\,(b_{2},B_2),\,\dots\, ,(b_{q},B_q);\end{array}  c(x \alpha)^2 \right]\nonumber\\
&&+(\alpha^2)  {_p\Psi_q}\left[\begin{array}{r} (a_{1},A_1),\,(a_{2},A_2),\,\dots\, ,(a_{p},A_p);\\~\\ (b_{1},B_1),\,(b_{2},B_2),\,\dots\, ,(b_{q},B_q);\end{array}  c(x \alpha^2)^2 \right]\nonumber\\
&&+ (\alpha^3) {_p\Psi_q}\left[\begin{array}{r} (a_{1},A_1),\,(a_{2},A_2),\,\dots\, ,(a_{p},A_p);\\~\\ (b_{1},B_1),\,(b_{2},B_2),\,\dots\, ,(b_{q},B_q);\end{array}  c(x \alpha^3)^2 \right]\nonumber\\
&&+  (\alpha^4) {_p\Psi_q}\left[\begin{array}{r} (a_{1},A_1),\,(a_{2},A_2),\,\dots\, ,(a_{p},A_p);\\~\\ (b_{1},B_1),\,(b_{2},B_2),\,\dots\, ,(b_{q},B_q);\end{array}  c(x \alpha^4)^2 \right]\nonumber\\
&&= \frac{ 5c^2x^4 \Gamma{\left(\frac{3}{5}\right)} \Gamma{\left(\frac{4}{5}\right)} \Gamma{\left(\frac{6}{5}\right)} \Gamma{\left(\frac{7}{5}\right)}}{2}\nonumber\\
&&\times {_p\Psi_{q+4}}\left[\begin{array}{r}  (a_{1}+2A_1,5A_1),\\~\\ \left(\frac{3}{5}, 1\right), \left(\frac{4}{5}, 1\right), \left(\frac{6}{5},1\right), \left(\frac{7}{5}, 1\right), (b_{1}+2B_1,5B_1), \end{array}  \right. \nonumber\\
&&\left.\begin{array}{r}\,(a_{2}+2A_2,5A_2),\,\dots\, ,(a_{p}+2A_p,5A_p);\\~\\
\,(b_{2}+2B_2,5B_2),\,\dots\, ,(b_{q}+2B_q,5B_q);\end{array} \left( \frac{cx^2}{5}\right)^5 \right],
\end{eqnarray}
where $\alpha=\exp{\left( \frac{2 \pi i}{5} \right)}$.
\end{theorem}

\begin{proof} On setting $\phi(r) = \frac{\displaystyle\prod_{j=1}^{p}\Gamma{(a_{j}+A_{j}r)}}{\displaystyle\prod_{j=1}^{q}\Gamma{(b_{j}+B_{j}r)}}$ in general series identities \eqref{h.eq(2.1)} and \eqref{h.eq(2.4)}, and applying the definition of Fox-Wright hypergeometric function $_p\Psi_q$, we get equations \eqref{h.eq(3.1)}, \eqref{h.eq(3.4)} respectively.
\end{proof}

\begin{theorem}
Following sum of special case of Fox-Wright hypergeometric functions with different arguments holds true:
\begin{eqnarray}\label{h.eq(3.5)}
&&{_p\Psi_{q}^{*}}\left[\begin{array}{r} (a_{1},A_1),\,(a_{2},A_2),\,\dots\, ,(a_{p},A_p);\\~\\ (b_{1},B_1),\,(b_{2},B_2),\,\dots\, ,(b_{q},B_q);\end{array}  cx^2 \right]\nonumber\\
&&+  {_p\Psi_{q}^{*}}\left[\begin{array}{r} (a_{1},A_1),\,(a_{2},A_2),\,\dots\, ,(a_{p},A_p);\\~\\ (b_{1},B_1),\,(b_{2},B_2),\,\dots\, ,(b_{q},B_q);\end{array}  c(x \alpha)^2 \right]\nonumber\\
&&+  {_p\Psi_{q}^{*}}\left[\begin{array}{r} (a_{1},A_1),\,(a_{2},A_2),\,\dots\, ,(a_{p},A_p);\\~\\ (b_{1},B_1),\,(b_{2},B_2),\,\dots\, ,(b_{q},B_q);\end{array}  c(x \alpha^2)^2 \right]\nonumber\\
&&+  {_p\Psi_{q}^{*}}\left[\begin{array}{r} (a_{1},A_1),\,(a_{2},A_2),\,\dots\, ,(a_{p},A_p);\\~\\ (b_{1},B_1),\,(b_{2},B_2),\,\dots\, ,(b_{q},B_q);\end{array}  c(x \alpha^3)^2 \right]\nonumber\\
&&+   {_p\Psi_{q}^{*}}\left[\begin{array}{r} (a_{1},A_1),\,(a_{2},A_2),\,\dots\, ,(a_{p},A_p);\\~\\ (b_{1},B_1),\,(b_{2},B_2),\,\dots\, ,(b_{q},B_q);\end{array}  c(x \alpha^4)^2 \right]\nonumber\\
&&= 5\,{_p\Psi_{q+4}^{*}}\left[\begin{array}{r} (a_{1},5A_1),\,(a_{2},5A_2),\,\dots\, ,(a_{p},5A_p);\\~\\ \left(\frac{1}{5}, 1\right), \left(\frac{2}{5}, 1\right), \left(\frac{3}{5}, 1\right), \left(\frac{4}{5}, 1\right), (b_{1},5B_1),\,(b_{2},5B_2),\,\dots\, ,(b_{q},5B_q);\end{array} \left( \frac{cx^2}{5}\right)^5 \right],\nonumber\\
\end{eqnarray}
where $\alpha=\exp{\left( \frac{2 \pi i}{5} \right)}.$
\end{theorem}

\begin{theorem}
Following sum of special case of Fox-Wright hypergeometric functions with different arguments holds true:
\begin{eqnarray}
&&{_p\Psi_{q}^{*}}\left[\begin{array}{r} (a_{1},A_1),\,(a_{2},A_2),\,\dots\, ,(a_{p},A_p);\\~\\ (b_{1},B_1),\,(b_{2},B_2),\,\dots\, ,(b_{q},B_q);\end{array}  cx^2 \right]\nonumber\\
&&+ (\alpha) {_p\Psi_{q}^{*}}\left[\begin{array}{r} (a_{1},A_1),\,(a_{2},A_2),\,\dots\, ,(a_{p},A_p);\\~\\ (b_{1},B_1),\,(b_{2},B_2),\,\dots\, ,(b_{q},B_q);\end{array}  c(x \alpha)^2 \right]\nonumber
\end{eqnarray}
\begin{eqnarray}\label{h.eq(3.6)}
&&+(\alpha^2)  {_p\Psi_{q}^{*}}\left[\begin{array}{r} (a_{1},A_1),\,(a_{2},A_2),\,\dots\, ,(a_{p},A_p);\\~\\ (b_{1},B_1),\,(b_{2},B_2),\,\dots\, ,(b_{q},B_q);\end{array}  c(x \alpha^2)^2 \right]\nonumber\\
&&+ (\alpha^3) {_p\Psi_{q}^{*}}\left[\begin{array}{r} (a_{1},A_1),\,(a_{2},A_2),\,\dots\, ,(a_{p},A_p);\\~\\ (b_{1},B_1),\,(b_{2},B_2),\,\dots\, ,(b_{q},B_q);\end{array}  c(x \alpha^3)^2 \right]\nonumber\\
&&+  (\alpha^4) {_p\Psi_{q}^{*}}\left[\begin{array}{r} (a_{1},A_1),\,(a_{2},A_2),\,\dots\, ,(a_{p},A_p);\\~\\ (b_{1},B_1),\,(b_{2},B_2),\,\dots\, ,(b_{q},B_q);\end{array}  c(x \alpha^4)^2 \right]\nonumber\\
&&= \frac{(a_1)_{2A_1}\dots(a_p)_{2A_p}}{(b_1)_{2B_1}\dots(b_q)_{2B_q}} \frac{5c^2 x^4}{2}\nonumber\\
&&\times {_p\Psi_{q+4}^{*}}\left[\begin{array}{r}  (a_{1}+2A_1,5A_1),\\~\\ \left(\frac{3}{5}, 1\right), \left(\frac{4}{5}, 1\right), \left(\frac{6}{5},1\right), \left(\frac{7}{5}, 1\right), (b_{1}+2B_1,5B_1),\end{array}  \right. \nonumber\\
&&\left.\begin{array}{r}\,(a_{2}+2A_2,5A_2),\,\dots\, ,(a_{p}+2A_p,5A_p);\\~\\
\,(b_{2}+2B_2,5B_2),\,\dots\, ,(b_{q}+2B_q,5B_q);\end{array} \left( \frac{cx^2}{5}\right)^5 \right],
\end{eqnarray}
where $\alpha=\exp{\left( \frac{2 \pi i}{5} \right)}.$
\end{theorem}

\begin{proof} On setting $\phi(r) = \frac{ (a_1)_{rA_1} (a_2)_{rA_2}\dots(a_p)_{rA_p}}{(b_1)_{rB_1} (b_2)_{rB_2}\dots (b_q)_{rB_q}}$ in general series identities \eqref{h.eq(2.1)} and \eqref{h.eq(2.4)}, and applying the definition of Fox-Wright hypergeometric function $_p\Psi^{*}_q$, we get equations \eqref{h.eq(3.5)}, \eqref{h.eq(3.6)} respectively.
\end{proof}

\begin{theorem}\label{h.thm3.8}
Following sum of generalized hypergeometric functions with different arguments holds true:
\begin{eqnarray}\label{h.eq(3.7)}
&&{_pF_{q}}\left[\begin{array}{r} (a_{p});\\~\\ (b_{q});\end{array}  cx^2 \right]
+  {_pF_{q}}\left[\begin{array}{r} (a_{p});\\~\\ (b_{q});\end{array}  c(x\alpha)^2 \right]+
 {_pF_{q}}\left[\begin{array}{r} (a_{p});\\~\\ (b_{q});\end{array}  c(x\alpha^2)^2 \right]\nonumber\\
&&+  {_pF_{q}}\left[\begin{array}{r} (a_{p});\\~\\ (b_{q});\end{array}  c(x\alpha^3)^2 \right]
+ {_pF_{q}}\left[\begin{array}{r} (a_{p});\\~\\ (b_{q});\end{array}  c(x\alpha^4)^2 \right]\nonumber\\
&&= 5\,{_{5p}F_{5q+4}}\left[\begin{array}{r} \frac{(a_p)}{5},\,\frac{1+(a_p)}{5},\,\frac{2+(a_p)}{5},\,\frac{3+(a_p)}{5},\, \frac{4+(a_p)}{5};\\~\\ \frac{1}{5},\frac{2}{5},\frac{3}{5},\frac{4}{5}, \frac{(b_q)}{5},\,\frac{1+(b_q)}{5},\,\frac{2+(b_q)}{5},\,\frac{3+(b_q)}{5},\, \frac{4+(b_q)}{5};\end{array} \left( \frac{cx^2}{5^{(1+q-p)}}\right)^5 \right],\nonumber\\
\end{eqnarray}
where $\alpha=\exp{\left( \frac{2 \pi i}{5} \right)}; 5p\leq5q+4, \Big|\left( \frac{cx^2}{5^{(1+q-p)}}\right)^5\Big|<\infty;  p=q+1, \Big|\left( \frac{cx^2}{5^{(1+q-p)}}\right)^5\Big|<1.$
\end{theorem}

\begin{theorem}\label{h.thm3.4}
Following sum of generalized hypergeometric functions with different arguments holds true:
\begin{eqnarray}
&& {_pF_{q}}\left[\begin{array}{r} (a_{p});\\~\\ (b_{q});\end{array}  cx^2 \right]
+ (\alpha) {_pF_{q}}\left[\begin{array}{r} (a_{p});\\~\\ (b_{q});\end{array}  c(x\alpha)^2 \right]+
 (\alpha^2){_pF_{q}}\left[\begin{array}{r} (a_{p});\\~\\ (b_{q});\end{array}  c(x\alpha^2)^2 \right]\nonumber
\end{eqnarray}
\begin{eqnarray}\label{h.eq(3.8)}
&&+  (\alpha^3){_pF_{q}}\left[\begin{array}{r} (a_{p});\\~\\ (b_{q});\end{array}  c(x\alpha^3)^2 \right]
+ (\alpha^4){_pF_{q}}\left[\begin{array}{r} (a_{p});\\~\\ (b_{q});\end{array}  c(x\alpha^4)^2 \right]\nonumber\\
&&= \frac{5c^2 x^4}{2} \frac{\prod_{i=1}^{p} (a_i)_2}{\prod_{i=1}^{q} (b_i)_2}  \,{_{5p}F_{5q+4}}\left[\begin{array}{r} \frac{2+(a_p)}{5},\,\frac{3+(a_p)}{5},\,\frac{4+(a_p)}{5},\,\frac{5+(a_p)}{5},\, \frac{6+(a_p)}{5};\\~\\ \frac{3}{5},\frac{4}{5},\frac{6}{5},\frac{7}{5}, \frac{2+(b_q)}{5},\,\frac{3+(b_q)}{5},\,\frac{4+(b_q)}{5},\, \frac{5+(b_q)}{5},\,\frac{6+(b_q)}{5};\end{array} \left( \frac{cx^2}{5^{(1+q-p)}}\right)^5 \right],\nonumber\\
\end{eqnarray}
where $\alpha=\exp{\left( \frac{2 \pi i}{5} \right)}; 5p\leq5q+4, \Big|\left( \frac{cx^2}{5^{(1+q-p)}}\right)^5\Big|<\infty;  p=q+1, \Big|\left( \frac{cx^2}{5^{(1+q-p)}}\right)^5\Big|<1.$
\end{theorem}

\begin{proof} On setting $\phi(r) = \frac{ (a_1)_{r} (a_2)_{r}\dots(a_p)_{r}}{(b_1)_{r} (b_2)_{r}\dots (b_q)_{r}}$ in general series identities \eqref{h.eq(2.1)} and \eqref{h.eq(2.4)}, and applying the definition of generalized hypergeometric function $_pF_q$, we get equations \eqref{h.eq(3.7)}, \eqref{h.eq(3.8)} respectively.
\end{proof}

\section{Applications}
As the direct application of our theorems 7 \& 8, we obtain following results on the sum of special functions and elementary functions with different arguments for $\alpha=\exp{\left( \frac{2 \pi i}{5} \right)}$:\\

On setting $p=0, q=1, c=\frac{-1}{4}$ and $b_1 = \frac{3}{2}$ in equation \eqref{h.eq(3.7)}, we obtain
\begin{eqnarray}\label{h.eq(4.1)}
&&\alpha^4\sin x + \alpha^3\sin (x\alpha) + \alpha^2\sin (x\alpha^2) + \alpha \sin (x\alpha^3) + \sin (x\alpha^4)\nonumber\\
&&\hskip20mm=5x\alpha^4\,{_{0}F_{9}}\left[\begin{array}{r} \text{\bf --------------------------};\\~\\ \frac{1}{5},\frac{2}{5},\frac{3}{5},\frac{4}{5},\frac{3}{10},\frac{5}{10},\frac{7}{10},\frac{9}{10},\frac{11}{10};\end{array} -\left( \frac{x}{10}\right)^{10} \right].
\end{eqnarray}
On setting $p=0, q=1, c=\frac{-1}{4}$ and $b_1 = \frac{3}{2}$ in equation \eqref{h.eq(3.8)}, we obtain
\begin{eqnarray}\label{h.eq(4.2)}
&&{\sin x} + \sin (x\alpha) + \sin (x\alpha^2) + \sin (x\alpha^3) + \sin (x\alpha^4)\nonumber\\
&&\hskip20mm= \frac{x^5}{24}\,\, {_{0}F_{9}}\left[\begin{array}{r} \text{\bf --------------------------};\\~\\ \frac{3}{5},\frac{4}{5},\frac{6}{5},\frac{7}{5},\frac{7}{10},\frac{9}{10},\frac{11}{10},\frac{13}{10},\frac{15}{10};\end{array} -\left( \frac{x}{10}\right)^{10} \right].
\end{eqnarray}
On setting $p=0, q=1, c=\frac{-1}{4}$ and $b_1 = \frac{1}{2}$ in equation \eqref{h.eq(3.7)}, we obtain
\begin{eqnarray}\label{h.eq(4.1)}
&&{\cos x} + \cos (x\alpha) + \cos (x\alpha^2) + \cos (x\alpha^3) + \cos (x\alpha^4)\nonumber\\
&&\hskip20mm =5\,{_{0}F_{9}}\left[\begin{array}{r} \text{\bf --------------------------};\\~\\ \frac{1}{5},\frac{2}{5},\frac{3}{5},\frac{4}{5},\frac{1}{10},\frac{3}{10},\frac{5}{10},\frac{7}{10},\frac{9}{10};\end{array} -\left( \frac{x}{10}\right)^{10} \right].
\end{eqnarray}
On setting $p=0, q=1, c=\frac{-1}{4}$ and $b_1 = \frac{1}{2}$ in equation \eqref{h.eq(3.8)}, we obtain
\begin{eqnarray}\label{h.eq(4.2)}
&&{\cos x} + (\alpha) \cos (x\alpha) + (\alpha^2)\cos (x\alpha^2) + (\alpha^3) \cos(\alpha^3) + (\alpha^4) \cos (\alpha^4)\nonumber\\
&&\hskip20mm = \frac{5x^4}{24}\,\, {_{0}F_{9}}\left[\begin{array}{r} \text{\bf --------------------------};\\~\\ \frac{3}{5},\frac{4}{5},\frac{6}{5},\frac{7}{5},\frac{5}{10},\frac{7}{10},\frac{9}{10},\frac{11}{10},\frac{13}{10};\end{array} -\left( \frac{x}{10}\right)^{10} \right].
\end{eqnarray}
On setting $p=2, q=1, c=-1$ and $a_1= \frac{1}{2},  a_2 =1,  b_1 = \frac{3}{2}$ in equation \eqref{h.eq(3.7)}, we obtain
\begin{eqnarray}\label{h.eq(4.5)}
&& \alpha^4 {\arctan x} + \alpha^3 {\arctan (x\alpha)} + \alpha^2{\arctan (x\alpha^2)} + \alpha {\arctan (x\alpha^3)} + {\arctan(x\alpha^4)}\nonumber\\
&&\hskip20mm=5x\alpha^4\,{_{2}F_{1}}\left[\begin{array}{r} 1, \frac{1}{10};\\~\\ \frac{11}{10};\end{array} -x^{10} \right].
\end{eqnarray}
On setting $p=2, q=1, c=-1$ and $a_1= \frac{1}{2},  a_2 =1,  b_1 = \frac{3}{2}$ in equation \eqref{h.eq(3.8)}, we obtain
\begin{eqnarray}\label{h.eq(4.6)}
&& {\arctan x} +  {\arctan (x\alpha)} + {\arctan (x\alpha^2)} + {\arctan (x\alpha^3)} + {\arctan (x\alpha^4)}\nonumber\\
&&\hskip20mm=x^5 \,{_{2}F_{1}}\left[\begin{array}{r} 1, \frac{1}{2};\\~\\ \frac{3}{2};\end{array} -x^{10} \right]={\arctan (x^5)}.
\end{eqnarray}
On setting $p=2, q=1, c=1$ and $a_1= \frac{1}{2},  a_2 =-\frac{1}{2},  b_1 = 1$ in equation \eqref{h.eq(3.7)}, we obtain
\begin{eqnarray}\label{h.eq(4.7)}
&&\textbf{E}(x) + \textbf{E}(x\alpha) + \textbf{E}(x\alpha^2) + \textbf{E}(x\alpha^3) + \textbf{E}(x\alpha^4)\nonumber\\
&&\hskip20mm= \left(\frac{5 \pi}{2}\right)\,{_{10}F_{9}}\left[\begin{array}{r} -\frac{1}{10},\frac{1}{10},\frac{1}{10},
 \frac{3}{10}, \frac{3}{10}, \frac{5}{10}, \frac{5}{10},\frac{7}{10}, \frac{7}{10}, \frac{9}{10};\\~\\ \frac{1}{5},\frac{1}{5},\frac{2}{5},\frac{2}{5},\frac{3}{5},\frac{3}{5},\frac{4}{5},\frac{4}{5},1;\end{array}  x^{10} \right].
\end{eqnarray}
On setting $p=2, q=1, c=1$ and $a_1= \frac{1}{2},  a_2 =-\frac{1}{2},  b_1 = 1$ in equation \eqref{h.eq(3.8)}, we obtain
\begin{eqnarray}\label{h.eq(4.8)}
&& \textbf{E}(x) + \alpha \textbf{E}(x\alpha) +\alpha^2 \textbf{E}(x\alpha^2) + \alpha^3 \textbf{E}(x\alpha^3) + \alpha^4 \textbf{E}(x\alpha^4)\nonumber\\
&&\hskip20mm= \left(-\frac{15x^4 \pi}{128}\right)\,{_{10}F_{9}}\left[\begin{array}{r}  \frac{3}{10}, \frac{5}{10},
\frac{15}{10}, \frac{7}{10}, \frac{7}{10}, \frac{9}{10}, \frac{9}{10}, \frac{11}{10},
\frac{11}{10}, \frac{13}{10};\\~\\
\frac{3}{5},\frac{3}{5},\frac{4}{5},\frac{4}{5},\frac{6}{5},\frac{6}{5},\frac{7}{5},\frac{7}{5},1;\end{array}  x^{10} \right].
\end{eqnarray}

On setting $p=1, q=1, c=-1$ and $a_1= \frac{1}{2}, b_1 = \frac{3}{2}$ in equation \eqref{h.eq(3.7)}, we obtain
\begin{eqnarray}\label{h.eq(4.9)}
&&\alpha^{4}\erf{(x)} + \alpha^{3}\erf{(x\alpha)} + \alpha^{2}\erf{(x\alpha^2)} + \alpha \erf{(x\alpha^3)} + \erf{(x\alpha^4)}\nonumber\\
&&\hskip20mm= \left(\frac{10x\alpha^{4}}{\sqrt{(\pi)}}\right)\,{_{1}F_{5}}\left[\begin{array}{r} \frac{1}{10};\\~\\
\frac{11}{10},\frac{1}{5},\frac{2}{5},\frac{3}{5},\frac{4}{5};\end{array}  -\left(\frac{x^{2}}{5}\right)^5 \right].
\end{eqnarray}
On setting $p=1, q=1, c=-1$ and $a_1= \frac{1}{2}, b_1 = \frac{3}{2}$ in equation \eqref{h.eq(3.8)}, we obtain
\begin{eqnarray}\label{h.eq(4.10)}
&& \erf{(x)} + \erf{(x\alpha)} + \erf{(x\alpha^2)} +  \erf{(x\alpha^3)} + \erf{(x\alpha^4)}\nonumber\\
&&\hskip20mm= \left(\frac{x^5}{\sqrt{(\pi)}}\right)\,{_{1}F_{5}}\left[\begin{array}{r} \frac{1}{2};\\~\\
\frac{3}{2},\frac{3}{5},\frac{4}{5},\frac{6}{5},\frac{7}{5};\end{array}  -\left(\frac{x^{2}}{5}\right)^5 \right].
\end{eqnarray}

On setting $p=3, q=2, c=1$ and $a_1= a_2= a_3=1$, $b_1 =2, b_2= \frac{3}{2}$ in equation \eqref{h.eq(3.7)}, we obtain
\begin{eqnarray}\label{h.eq(4.11)}
&&\alpha^8 (\arcsin x)^2 + \alpha^6 (\arcsin (x\alpha))^2 + \alpha^4 (\arcsin (x\alpha^2))^2 + \alpha^2 (\arcsin (x\alpha^3))^2 + (\arcsin (x\alpha^4))^2\nonumber\\
&&\hskip20mm=5x^2\alpha^8\,{_{7}F_{6}}\left[\begin{array}{r} \frac{1}{5},\frac{1}{5},\frac{2}{5},\frac{3}{5},\frac{4}{5},1,1;\\~\\ \frac{6}{5},\frac{3}{10},\frac{1}{2},\frac{7}{10},\frac{9}{10},\frac{11}{10};\end{array} x^{10} \right].
\end{eqnarray}
On setting $p=3, q=2, c=1$ and $a_1= a_2= a_3=1$, $b_1 =2, b_2= \frac{3}{2}$ in equation \eqref{h.eq(3.8)}, we obtain
\begin{eqnarray}\label{h.eq(4.12)}
&&\alpha^4 (\arcsin x)^2 + \alpha^3 (\arcsin (x\alpha))^2 + \alpha^2 (\arcsin (x\alpha^2))^2 + \alpha (\arcsin (x\alpha^3))^2 + (\arcsin (x\alpha^4))^2\nonumber\\
&&\hskip20mm=\frac{8x^6\alpha^4}{9}\,{_{7}F_{6}}\left[\begin{array}{r} \frac{3}{5},\frac{3}{5},\frac{4}{5},\frac{6}{5},\frac{7}{5},1,1;\\~\\ \frac{8}{5},\frac{7}{10},\frac{9}{10},\frac{11}{10},\frac{13}{10},\frac{3}{2};\end{array} x^{10} \right].
\end{eqnarray}

On setting $p=2, q=1, c=1$ and $a_1= a_2 =\frac{1}{2},  b_1 = 1$ in equation \eqref{h.eq(3.7)}, we obtain
\begin{eqnarray}\label{h.eq(4.13)}
&&\textbf{K}(x) + \textbf{K}(x\alpha) + \textbf{K}(x\alpha^2) + \textbf{K}(x\alpha^3) + \textbf{K}(x\alpha^4)\nonumber\\
&&\hskip20mm= \left(\frac{5 \pi}{2}\right)\,{_{10}F_{9}}\left[\begin{array}{r} \frac{1}{10},\frac{1}{10}, \frac{3}{10}, \frac{3}{10}, \frac{1}{2}, \frac{1}{2},\frac{7}{10}, \frac{7}{10}, \frac{9}{10}, \frac{9}{10};\\~\\ \frac{1}{5},\frac{1}{5},\frac{2}{5},\frac{2}{5},\frac{3}{5},\frac{3}{5},\frac{4}{5},\frac{4}{5},1;\end{array}  x^{10} \right].
\end{eqnarray}
On setting $p=2, q=1, c=1$ and $a_1= a_2=\frac{1}{2}, b_1 = 1$ in equation \eqref{h.eq(3.8)}, we obtain
\begin{eqnarray}\label{h.eq(4.14)}
&& \textbf{K}(x) + \alpha \textbf{K}(x\alpha) +\alpha^2 \textbf{K}(x\alpha^2) + \alpha^3 \textbf{K}(x\alpha^3) + \alpha^4 \textbf{K}(x\alpha^4)\nonumber\\
&&\hskip20mm= \left(\frac{45 x^4 \pi}{128}\right)\,{_{10}F_{9}}\left[\begin{array}{r} \frac{1}{2},\frac{1}{2},\frac{7}{10}, \frac{7}{10}, \frac{9}{10}, \frac{9}{10}, \frac{11}{10}, \frac{11}{10},\frac{13}{10}, \frac{13}{10};\\~\\ \frac{3}{5},\frac{3}{5},\frac{4}{5},\frac{4}{5},\frac{6}{5},\frac{6}{5},\frac{7}{5},\frac{7}{5},1;\end{array}  x^{10} \right].
\end{eqnarray}

On setting $p=3, q=2, c=1$ and $a_1= a_2 =a_3=1,  b_1 =b_2= 2$ in equation \eqref{h.eq(3.7)}, we obtain
\begin{eqnarray}\label{h.eq(4.15)}
&&\alpha^8Li_2(x^2) + \alpha^6Li_2((x\alpha)^2) + \alpha^4Li_2((x\alpha^2)^2) + \alpha^2Li_2((x\alpha^3)^2) + Li_2((x\alpha^4)^2)\nonumber\\
&&\hskip20mm= \left(5x^2\alpha^8\right)\,{_{3}F_{2}}\left[\begin{array}{r} \frac{1}{5},\frac{1}{5}, 1;\\~\\ \frac{6}{5},\frac{6}{5};\end{array}  x^{10} \right].
\end{eqnarray}
On setting $p=3, q=2, c=1$ and $a_1= a_2 =a_3=1,  b_1 =b_2= 2$ in equation \eqref{h.eq(3.8)}, we obtain
\begin{eqnarray}\label{h.eq(4.16)}
&&\alpha^4Li_2(x^2) + \alpha^3Li_2((x\alpha)^2) + \alpha^2Li_2((x\alpha^2)^2) + \alpha Li_2((x\alpha^3)^2) + Li_2((x\alpha^4)^2)\nonumber\\
&&\hskip20mm= \left(\frac{5x^6\alpha^4}{9}\right)\,{_{3}F_{2}}\left[\begin{array}{r} \frac{3}{5},\frac{3}{5}, 1;\\~\\ \frac{8}{5},\frac{8}{5};\end{array}  x^{10} \right].
\end{eqnarray}

On setting $p=1, q=1, c=-1$ and $a_1= a, b_1 = 1+a$ in equation \eqref{h.eq(3.7)}, we obtain
\begin{eqnarray}\label{h.eq(4.17)}
&& \alpha^{8a}\gamma(a,x^2) + \alpha^{6a}\gamma(a,(x\alpha)^2) +\alpha^{4a}\gamma(a,(x\alpha^2)^2) + \alpha^{2a}\gamma(a,(x\alpha^3)^2) + \gamma(a,(x\alpha^4)^2)\nonumber\\
&&\hskip20mm= \left(\frac{5 x^{2a}\alpha^{8a}}{a}\right)\,{_{1}F_{5}}\left[\begin{array}{r} \frac{a}{5};\\~\\ \frac{a+5}{5},\frac{1}{5},\frac{2}{5},\frac{3}{5},\frac{4}{5};\end{array}  -\left(\frac{x^2}{5}\right)^5 \right].
\end{eqnarray}

On setting $p=1, q=1, c=-1$ and $a_1= a, b_1 = 1+a$ in equation \eqref{h.eq(3.8)}, we obtain
\begin{eqnarray}\label{h.eq(4.18)}
&& \alpha^{8a}\gamma(a,x^2) + \alpha^{6a+1}\gamma(a,(x\alpha)^2) +\alpha^{4a+2}\gamma(a,(x\alpha^2)^2) + \alpha^{2a+3}\gamma(a,(x\alpha^3)^2) +\alpha^4 \gamma(a,(x\alpha^4)^2)\nonumber\\
&&\hskip20mm= \left(\frac{5 x^{2a+4}\alpha^{8a}}{2(a+2)}\right)\,{_{1}F_{5}}\left[\begin{array}{r} \frac{a+2}{5};\\~\\ \frac{a+7}{5},\frac{3}{5},\frac{4}{5},\frac{6}{5},\frac{7}{5};\end{array}  -\left(\frac{x^2}{5}\right)^5 \right].
\end{eqnarray}

On setting $p=2, q=1, c=1$ and $a_1= \gamma, a_2= \gamma-\frac{1}{2}, b_1 = 2\gamma$ in equation \eqref{h.eq(3.7)}, we obtain
\begin{eqnarray}\label{h.eq(4.17)}
&&\left(\frac{2}{1+\sqrt{(1-(x)^2)}}\right)^{2\gamma-1} + \left(\frac{2}{1+\sqrt{(1-(x\alpha)^2)}}\right)^{2\gamma-1} + \left(\frac{2}{1+\sqrt{(1-(x\alpha^2)^2)}}\right)^{2\gamma-1}\nonumber\\
&&+ \left(\frac{2}{1+\sqrt{(1-(x\alpha^3)^2)}}\right)^{2\gamma-1} + \left(\frac{2}{1+\sqrt{(1-(x\alpha^4)^2)}}\right)^{2\gamma-1}\nonumber\\
&&\hskip5mm=5\,{_{10}F_{9}}\left[\begin{array}{r} \frac{\gamma}{5},\frac{\gamma+1}{5},\frac{\gamma+2}{5},\frac{\gamma+3}{5},\frac{\gamma+4}{5},\frac{2\gamma-1}{10},
\frac{2\gamma+1}{10},\frac{2\gamma+3}{10},\frac{2\gamma+5}{10},\frac{2\gamma+7}{10};\\~\\ \frac{1}{5},\frac{2}{5},\frac{3}{5},\frac{4}{5},\frac{2\gamma}{5},\frac{2\gamma+1}{5},\frac{2\gamma+2}{5},\frac{2\gamma+3}{5},\frac{2\gamma+4}{5};\end{array}  x^{10} \right].
\end{eqnarray}

On setting $p=2, q=1, c=1$ and $a_1= \gamma, a_2= \gamma-\frac{1}{2}, b_1 = 2\gamma$ in equation \eqref{h.eq(3.8)}, we obtain
\begin{eqnarray}\label{h.eq(4.17)}
&&\left(\frac{2}{1+\sqrt{(1-(x)^2)}}\right)^{2\gamma-1} + \alpha\left(\frac{2}{1+\sqrt{(1-(x\alpha)^2)}}\right)^{2\gamma-1} + \alpha^2\left(\frac{2}{1+\sqrt{(1-(x\alpha^2)^2)}}\right)^{2\gamma-1}\nonumber\\
&&+ \alpha^3 \left(\frac{2}{1+\sqrt{(1-(x\alpha^3)^2)}}\right)^{2\gamma-1} + \alpha^4 \left(\frac{2}{1+\sqrt{(1-(x\alpha^4)^2)}}\right)^{2\gamma-1}\nonumber\\
&&=\left(\frac{5 (\gamma)_2 (\gamma-\frac{1}{2})_2 \, x^4}{2 (2\gamma)_2}\right)\,{_{10}F_{9}}\left[\begin{array}{r} \frac{\gamma+2}{5},\frac{\gamma+3}{5},\frac{\gamma+4}{5},\frac{\gamma+5}{5},\frac{\gamma+6}{5},
\frac{2\gamma+3}{10},\frac{2\gamma+5}{10},\frac{2\gamma+7}{10},\frac{2\gamma+9}{10},
\frac{2\gamma+11}{10};\\~\\ \frac{3}{5},\frac{4}{5},\frac{6}{5},\frac{7}{5},
\frac{2\gamma+2}{5},\frac{2\gamma+3}{5},\frac{2\gamma+4}{5},\frac{2\gamma+5}{5},\frac{2\gamma+6}{5};\end{array}  x^{10} \right].\nonumber\\
\end{eqnarray}

\vskip.2cm
{\bf Remark:} Making suitable adjustments of parameters and variables in Theorems 3,4,5 and 6, we can derive some more results involving generalised Bessel functions $\Phi(\alpha, \beta; z) \text{ or } J^{\mu}_{\nu}(z)$, Mittag-Leffler functions $E_{\alpha}(z)$ and its generalizations $E_{\alpha, \beta}(z)$.
Since Wright's generalized function $_p\Psi_q$ of one variable is the particular case of Fox \textit{H}-function of one variable. Therefore, for more special cases of $_p\Psi_q$, we refer two monographs of Mathai-Saxena \cite{Mathai} and  Srivastava, Gupta and Goyal \cite{Srivastava1}.

\section{Conclusion}
Here, we have established some results on the sum of hypergeometric
functions with different arguments. We applied our results to Trigonometric functions, Elliptic integrals, Dilogarithmic function, Error function, and Incomplete gamma function.\\

One can also establish the some results on the sum of these special functions for example:
ordinary Bessel function J$_{\nu}$(x), modified Bessel function I$_{\nu}$(x),
complete elliptic integrals {\bf B}(x), {\bf C}(x), {\bf D}(x),
Lerch's transcendent $\Phi(x,q,a)$, Fresnel's integrals S(x), S$_1$(x),
S$_2$(x), C$^{*}$(x), C$_{1}$(x), C$_2$(x),
Sine integral S$_i$(x), hyperbolic sine integral Sh$_i$(x),
Polylogarithm function $Li_{q}$(x),
Sturve function H$_{\nu}$(x),
Modified Sturve functions {\bf L}$_{\nu}$(x), h$_{\mu,\nu}$(x),
Lommel function s$_{\mu,\nu}$(x),
Kelvin's functions $\text{ ber}(x),\text{ bei(x)}$,
Incomplete beta function B$_x(\alpha, \beta)$,
Hyperbessel function of Humbert J$_{m,n}$(x),
Modified hyperbessel function of Delerue I$_{m,n}$(x),
Arctangents function Ti$_2$(x), $\sinh(x), \cosh(x), \tanh^{-1}(x), \sinh^{-1}(x)$, \\ $\frac{\sinh^{-1}(x)}{\sqrt{\left(1+x^2\right)}}$, $\left(\sinh^{-1}x\right)^2$, $\sin{\left(a\sin^{-1}x\right)}$, $\cos{\left(a\sin^{-1}x\right)}$, $\log_{e}(1\pm x)$, $\frac{\arcsin(x)}{\sqrt{(1-x^2)}}$ and $\exp{(a\sin^{-1}x)}$ etc.


\begin{thebibliography}{99}
\bibitem{Boersma} Boersma, J.; On a function which is a special case of Meijer's \textit{G}-function, {\it Compositio Math.}, {\bf 15} (1962), 34-63.

\bibitem{Braaksma} Braaksma, B.L.J.; Asymptotic expansions and analytic continuations for a class of Barnes-integrals, {\it Compositio Math.}, {\bf 15} (1964), 239-341.

\bibitem{Fox1} Fox, C.; The asymptotic expansion of generalized hypergeometric functions, {\it Proc. London Math. Soc.}, {\bf 27}(2) (1928), 389-400.

\bibitem{Fox2} Fox, C.; The \textit{G} and \textit{H} functions as symmetrical Fourier kernels, \textit{Trans. Amer. Math. Soc.}, {\bf 98} (1961), 395-421.
\bibitem{Gradshteyn} Gradshteyn, I.S. and Ryzhik, I.M.;  \emph{Table of integrals, series and products}, 8th ed., Academic Press Inc., San Diego, CA. 2014.

\bibitem{Mathai} Mathai, A.M. and Saxena, R.K.;  \emph{The \textit{H}-function with applications in statistics and other disciplines}, John Wiley and Sons (Halsted Press), New York, 1978.

\bibitem{Rainville}Rainville, E.D.; \textit{Special Functions,} The Macmillan Co. Inc.,New York,1960;Reprinted by Chelsea Publ. Co. Bronx, New York, 1971.

\bibitem{Srivastava1} Srivastava, H.M. Gupta, K.C. and Goyal, S.P.;\textit{ The \textit{H}-functions of one and two variables with applications,} South Asian Publishers, New Delhi and Madras, 1982.

\bibitem{Srivastava3} Srivastava, H.M. and Manocha, H.L.;\textit{ A Treatise on Generating functions,} Halsted Press (Ellis Horwood Ltd., Chichester, U.K.), John Wiley and Sons, New York, Chichester, Brisbane and Toronto, 1984.

\bibitem{Wright1} Wright, E.M.; The asymptotic expansion of the generalized hypergeometric function- I, \textit{J. London Math. Soc.}, {\bf 10}(4) (1935), 286-293.

\bibitem{Wright2} Wright, E.M.; The asymptotic expansion of the generalized hypergeometric function-II, {\it Proc. London Math. Soc.(2)}, {\bf 46} (1940), 389-408.

\end{thebibliography}
\end{document}